\documentclass[11pt]{article}
\usepackage[top=1in, bottom=1in, left=1in,right=1in]{geometry}
\usepackage{amsmath}
\usepackage{amssymb}
\usepackage{titlesec}
\usepackage{color}
\usepackage{amsthm}
\usepackage{mathrsfs}   
\usepackage{theoremref}
\usepackage{stix}
\usepackage{hyperref}
\usepackage{graphicx}

\usepackage{biblatex}
\renewbibmacro{in:}{} 

\usepackage[all]{xy}

\newtheorem{defn}{Definition}[section]
\newtheorem{thm}[defn]{Theorem}
\newtheorem*{thm*}{\textsc{Theorem}}
\newtheorem{lem}[defn]{Lemma}  
\newtheorem{prop}[defn]{Proposition}

\theoremstyle{definition}
\newtheorem*{rmk}{Remark}
\theoremstyle{definition}

\title{An Exact Counting Formula for the Mutual Position of Two Plane Conics}
\author{Tianhao Wang\thanks{The author was supported in part by NSF Grant DMS-2154223.}\thanks{Email: \texttt{tianhw11@uci.edu}}}
\date{\today}

\begin{document}
\maketitle
\begin{abstract}
Let $q$ be an odd prime power, and $\mathcal{C}, \mathcal{D}$ be two smooth plane conics defined over $\mathbb{F}_q$ with transversal intersection. We present an exact formula for the number of points in $\mathbb{P}^2(\mathbb{F}_q)$ that are internal/external to $\mathcal{C}$ and internal/external to $\mathcal{D}$. This refines the $\frac{q^2}{4}+O(q^{3/2})$ asymptotic estimate for this quantity due to Asgarli and Yip \cite[Theorem 1.2]{Asgarli}.  In particular, we show that the error term is of size at most $q+\sqrt{q} + 1$. By studying the geometry of the incidence variety related to this problem, we link the exact point counts directly to the Frobenius traces of two associated elliptic curves, and the number of $\mathbb{F}_q$-rational intersection points of $\mathcal{C}$ and $\mathcal{D}$ and of the corresponding dual conics $\mathcal{C}^*$ and $\mathcal{D}^*$. Lastly, We provide a remark explaining the challenges in generalizing this method to the study of higher-dimensional quadrics.
\end{abstract}


\section{Introduction}
Let $q$ be an odd prime power and $\mathcal{C}, \mathcal{D} \subset \mathbb{P}^2$ be two smooth plane conics defined over $\mathbb{F}_q$ with transversal intersection. Following classical geometric terminology over finite fields, the configuration of an external or internal point relative to a smooth conic is characterized by the rationality of its tangent lines.

\begin{defn}
Let \(P \in \mathbb P^2(\mathbb F_q)\setminus \mathcal C\). Over $\overline{\mathbb{F}_q}$, there are two distinct tangent lines of $\mathcal{C}$ passing through $P$. There are two possibilities: 
\begin{itemize}
  \item We say that \(P\) is \textbf{external} to \(\mathcal C\) if those two tangent lines are defined over \(\mathbb F_q\).
  \item We say that \(P\) is \textbf{internal} to \(\mathcal C\) if those two tangent lines are Galois-conjugate lines defined over \(\mathbb F_{q^2}\), but not over $\mathbb{F}_q$. 
\end{itemize}
\end{defn}

A problem originally posed by Korchmáros \cite[Problem 6.2]{Korchmros} asks for the number of points in the projective plane that occupy specific mutual positions relative to a pair of conics. Motivated by this query, Asgarli and Yip \cite[Theorem 1.2]{Asgarli} proved that the number of points in $\mathbb{P}^2(\mathbb{F}_q)$ that are simultaneously external to $\mathcal{C}$ and internal to $\mathcal{D}$ is given by 
\[
\frac{q^2}{4} + O\left(q^{3/2}\right).
\]
More generally, they established analogous asymptotic estimates for point distributions relative to pairs of smooth quadric hypersurfaces in higher dimensions, using asymptotic estimates for multiplicative character sums. They \cite[Theorem 1.3]{Asgarli} proved that for an even integer $n\geq 2$, the number of points in $\mathbb{P}^{n}(\mathbb{F}_q)$ that are internal/external to $\mathcal{C}$ and internal/external to $\mathcal{D}$ is given by
\[
\frac{q^{n}}{4} + O\left(q^{n-1/2}\right).
\]

Shortly after, Slavov \cite{Slavov} reinterpreted this setup through the joint distribution of square values of several polynomials over a finite field. By treating the condition of being external/internal as the evaluation of squares or non-squares of the defining quadratic forms, Slavov applied the Lang–Weil bound to a corresponding variety. This geometric approach provides a simpler proof of the Asgarli–Yip bounds, and also improved the explicit constants in the error term.

In this paper, we will give an exact formula for the number of $\mathbb{F}_q$-rational points that are internal/external to $\mathcal{C}$ and internal/external to $\mathcal{D}$. To state our main result, let $C$ and $D$ be the symmetric matrix representations of the plane conics $\mathcal{C}$ and $\mathcal{D}$. We introduce two elliptic curves associated to the conic pair and their duals:
\[
E_1: y^2 = \det(D)\det(xC + D), \quad \text{and} \quad E_2: y^2 = \det(C)\det(xD + C).
\]
Let $t_1$ and $t_2$ denote the Frobenius traces of $E_1$ and $E_2$ respectively. Furthermore, let $b_1$ be the number of $\mathbb{F}_q$-rational intersection points of $\mathcal{C}$ and $\mathcal{D}$, and let $b_2$ be the number of $\mathbb{F}_q$-rational intersection points of the dual conics $\mathcal{C}^*$ and $\mathcal{D}^*$. Our main result is formulated as follows:

\begin{thm*}[Theorem \ref{thm: main_conic_case}]
Suppose that $q$ is an odd prime power, and $\mathcal{C}$, $\mathcal{D}$ are two smooth plane conics with transversal intersection. Then, the number of points in $\mathbb{P}^2(\mathbb{F}_q)$ external to $\mathcal{C}$ and internal to $\mathcal{D}$ is given by
\[
 N_{\mathrm{ext\text{-}int}}(\mathcal{C}, \mathcal{D}) = \frac{1}{4}({q^2 - b_2 q - 1 - t_1 + b_1 + t_2}).
\]
Moreover, we also compute that
\begin{align*}
    N_{\mathrm{int\text{-}ext}}(\mathcal C,\mathcal D) &= \frac{1}{4}(q^2-b_2q-1+t_1+b_1-t_2), \\
    N_{\mathrm{ext\text{-}ext}}(\mathcal C,\mathcal D) &= \frac{1}{4}(q^2+b_2q-1+t_1+b_1+t_2), \\
    N_{\mathrm{int\text{-}int}}(\mathcal C,\mathcal D) &= \frac{1}{4}(q^2+b_2q-1-t_1+b_1-t_2-4q).
\end{align*}
In particular, we have 
$$\left|N_{\mathrm{int/ext\text{-}int/ext}}(\mathcal{C}, \mathcal{D}) - \frac{q^2}{4}\right|\leq q+\sqrt{q}+1.$$
\end{thm*}

As a byproduct of this calculation, we also obtain exact formulas for the number of $\mathbb F_q$-rational points of $\mathcal{C}$ that are internal/external to $\mathcal{D}$. This refines the bounds appearing in \cite[Theorem~1.1]{abatangelo2011mutual} under the transversality hypothesis considered here.

\begin{prop}
Assume that $\mathcal C$ and $\mathcal D$ intersect transversely. Let $E_1: y^2=(\det D)\det(xC+D)$ be the elliptic curve associated to the pair $(\mathcal C,\mathcal D)$ and $t_1$ be its Frobenius trace. Let $b_1$ be the number of $\mathbb{F}_q$-rational intersection points of $\mathcal{C}$ and $\mathcal D$. Then
\begin{align*}
|\{P\in \mathcal C(\mathbb F_q)\setminus \mathcal D : P \text{ is external to } \mathcal D \}|
&= \frac{1}{2}(q+1-t_1-b_1), \\
|\{P\in \mathcal C(\mathbb F_q)\setminus \mathcal D : P \text{ is internal to } \mathcal D \}|
&= \frac{1}{2}(q+1+t_1-b_1).
\end{align*}
\end{prop}

Our proof is geometric and relies on analyzing two incidence varieties associated with the conic pair. While the broad higher-dimensional asymptotic machinery via multi-quadratic covers and Chebotarev-type equidistribution is established by Slavov in \cite{Slavov} and also our companion work \cite{Mypaper}, the planar conic setting stands out as the error terms can be computed explicitly. We will conclude this paper with a detailed remark explaining why it is hard to extend this method to higher-dimensional quadrics.

\section{Proof of the Main Counting Formula}
\subsection*{Idea of the approach}
We pick an \(\mathbb F_q\)-rational tangent line \(\mathcal L\) to \(\mathcal C\).  
Every \(\mathbb F_q\)-point on this line other than the tangency point is, by definition, external to \(\mathcal C\).  
Hence, to count points in \(\mathbb P^2(\mathbb F_q)\) that are external to \(\mathcal C\) and internal to \(\mathcal D\), it suffices to understand, for each such tangent line \(\mathcal L\), how many points on \(\mathcal L\) are internal to \(\mathcal D\).

For a fixed $\mathbb{F}_q$-tangent line \(\mathcal L\) to $\mathcal{C}$, there are three possibilities:
\begin{enumerate}
  \item[\textbf{(1)}] \(\mathcal L\cap\mathcal D\) consists of two distinct \(\mathbb F_q\)-points.
  \item[\textbf{(2)}] \(\mathcal L\cap\mathcal D\) consists of two \(\mathbb F_{q^2}\)-conjugate points.
  \item[\textbf{(3)}] \(\mathcal L\) is also tangent to \(\mathcal D\) (a \emph{bad} line).
\end{enumerate}
For cases (1) and (2), we consider the incidence variety
\begin{equation}\label{eq:incidence} 
\{(P,T)\in \mathcal L\times \mathcal D^* : P\in T\}\to \mathcal L
\end{equation}
as a double cover branched at the two intersection points \(\mathcal L\cap\mathcal D\). Then, the number of $\mathbb{F}_q$-rational points on $\mathcal{L}$ that are internal to $\mathcal{D}$ depends on whether the two branch points are $\mathbb{F}_q$-rational or $\mathbb{F}_{q^2}$-conjugate.
Concretely:
\begin{itemize}
  \item \textbf{Case 1}: If the branch points $\mathcal L\cap \mathcal D$ are \(\mathbb F_q\)-rational, there are \(\tfrac{q-1}{2}\) points on \(\mathcal L\) internal to $\mathcal{D}$;
  \item \textbf{Case 2}: If the branch points $\mathcal L\cap\mathcal D$ are \(\mathbb F_{q^2}\)-conjugate, there are \(\tfrac{q+1}{2}\) points on $\mathcal{L}$ internal to $\mathcal{D}$.
\end{itemize}
The remaining “bad” tangent lines (case (3)), those also tangent to \(\mathcal D\), contribute only external points to the conic \(\mathcal D\). The number of such bad lines equals the number of \(\mathbb F_q\)-points of \(\mathcal C^*\cap\mathcal D^*\).

Now, we prove the counts appearing in the above cases (1) and (2). 
\begin{proof}
    By the Riemann-Hurwitz formula, defined in \eqref{eq:incidence} has genus $0$, and thus $q+1$  $\mathbb{F}_q$-rational points. Each $P\in\mathcal{L}(\mathbb{F}_q)$ that is external to $\mathcal{D}$ will contribute two $\mathbb{F}_q$-rational points on the incidence variety, and the points internal to $\mathcal{D}$ contribute none. If the two branch points $\mathcal{L}\cap \mathcal{D}$ are $\mathbb{F}_q$-rational, they will each contribute one $\mathbb{F}_q$-rational points on the incidence variety. 

    Therefore, in case (1), the $q+1$ rational points on the incidence variety comes from 
    $$q+1 = 2 + 2\cdot \frac{q-1}{2},$$
    where the $q+1$ points on $\mathcal{L}(\mathbb{F}_q)$ split into $2$ branch points, $\frac{q-1}{2}$ points external to $\mathcal{D}$, and $\frac{q-1}{2}$ points internal to $\mathcal{D}$. 

    In case (2), we have 
    $$q+1 = 2\cdot \frac{q+1}{2},$$
    where the $q+1$ points on $\mathcal{L}(\mathbb{F}_q)$ splits into $\frac{q+1}{2}$ points external to $\mathcal{D}$, and $\frac{q+1}{2}$ points internal to $\mathcal{D}$. 
\end{proof}

\subsection*{Elliptic curves and Frobenius traces}

When we let the tangent line \(\mathcal L\) vary over all \(\mathbb F_q\)-tangent lines of \(\mathcal C\), the pattern of how \(\mathcal L\) intersects \(\mathcal D\) is controlled by an elliptic curve naturally associated to the conic pair \((\mathcal C,\mathcal D)\).  

We first prove a lemma that if the conics $\mathcal{C}$ and $\mathcal{D}$ intersect transversally, then so do the corresponding dual conics, $\mathcal{C}^*$ and $\mathcal{D^*}$. 
\begin{lem}\label{lem:dual-transversality}
Let \(k\) be a field of characteristic not equal to \(2\), and let
\(\mathcal C,\mathcal D\subset \mathbb P^2_k\) be smooth conics represented by symmetric matrices \(C\) and \(D\). If \(\mathcal C\) and
\(\mathcal D\) intersect transversely over \(\overline{k}\), then their dual
conics \(\mathcal C^*\) and \(\mathcal D^*\) also intersect transversely over
\(\overline{k}\).
\end{lem}

\begin{proof}
We refer to a classic result \cite[Proposition 22.34]{harrisag} that the conics $\mathcal{C}$ and $\mathcal{D}$ intersect transversally if and only if
the discriminant polynomial 
$$\det(xC+zD)$$
has $3$ distinct roots over 
\(\mathbb{P}^1(\overline{k})\). Since the dual conics are represented by \(C^{-1}\) and \(D^{-1}\), their discriminant
polynomial is
\[
\det(xC^{-1}+zD^{-1}).
\]
Using
\[
\det(xC^{-1}+zD^{-1})
=
(\det C\det D)^{-1}\det(xD+zC),
\]
we see that this polynomial has $3$ distinct roots if and only if
\(\det(xC+zD)\) has $3$ distinct roots. Hence \(\mathcal C^*\) and
\(\mathcal D^*\) intersect transversely if and only if $\mathcal C$ and $\mathcal D$ intersect transversally.
\end{proof}

Now, we assume that the conics $\mathcal{C}, \mathcal{D}$ intersect transversally. Based on the above lemma, their duals $\mathcal{C}^*$ and $\mathcal{D}^*$ also intersect transversally. We introduce the incidence varieties
$$E_1 = \{(P, \mathcal L)\in \mathcal C\times\mathcal D^* : P\in \mathcal L\}, \quad E_2 = \{(\mathcal L, P)\in\mathcal{C}^*\times\mathcal{D}: P\in\mathcal{L}\}.$$
The projection \(E_1\to\mathcal C\) is a $2$-to-$1$ map branched at the four intersection points of $\mathcal{C}$ and $\mathcal{D}$ over $\overline{\mathbb{F}_q}$. By Riemann-Hurwitz, \(E_1\) has genus one, and it is an elliptic curve. For the same reason, $E_2$ is also an elliptic curve by considering the 2-to-1 projection map $E_2\to \mathcal{C}^*$ ramified at the four intersection points of $\mathcal{C}^*$ and $\mathcal{D}^*$.

We will prove later in Lemma \ref{lem:incidence-cayley-twist} that $E_1$ and $E_2$ can be represented over $\mathbb{F}_q$ by the models
$$E_1: y^2 = \det(D)\det(xC+D), \quad
E_2: y^2 = \det(C)\det(xD+C).$$

If \(P\in \mathcal C(\mathbb F_q)\setminus\mathcal D\) is external to $\mathcal{D}$, it will contribute
two \(\mathbb F_q\)-points to \(E_1(\mathbb F_q)\). Else, it contributes none. The $\mathbb{F}_q$-rational branch points \(P\in(\mathcal C\cap\mathcal D)(\mathbb F_q)\) contribute one point each on $E_1(\mathbb{F}_q)$. Therefore, we have 
$$|\{P\in \mathcal{C}(\mathbb{F}_q)\setminus\mathcal{D}: P \text{ is external to } \mathcal{D} \}|=\frac{|E_1(\mathbb F_q)|-b_1}{2}=\frac{1}{2}(q+1-t_1-b_1),$$
where $t_1$ is the trace of Frobenius of the elliptic curve $E_1$, and $b_1$ is the number of $\mathbb{F}_q$-rational points of $\mathcal{C}\cap\mathcal{D}$. Since \(|\mathcal C(\mathbb F_q)|=q+1\), the number of points in $\mathcal{C}(\mathbb{F}_q)$ that are internal to $\mathcal{D}$ is
$$|\{P\in \mathcal{C}(\mathbb{F}_q)\setminus\mathcal D: P \text{ is internal to } \mathcal{D} \}| = q+1-b_1-\frac{q+1-t_1-b_1}{2} = \frac{1}{2}(q+1+t_1-b_1).$$

Similarly, in the dual setting, let $t_2$ be the trace of Frobenius of $E_2$, and $b_2$ be the number of $\mathbb{F}_q$-rational points of $C^*\cap D^*$. We can count the number of tangent lines of $\mathcal{C}$ that are in case (1), (2), (3) defined in the previous section by
\begin{align*}
    |\{\mathcal L\in \mathcal{C}^*(\mathbb{F}_q)\setminus\mathcal{D}^*: \mathcal L\cap \mathcal{D} \text{ consists of two }\mathbb{F}_q\text{-points}\}| &= \frac{|E_2(\mathbb{F}_q)|-b_2}{2} = \frac{1}{2}(q+1-t_2-b_2), \\
    |\{\mathcal L\in \mathcal{C}^*(\mathbb{F}_q)\setminus\mathcal D^*: \mathcal L\cap \mathcal{D} \text{ consists of two }\mathbb{F}_{q^2}\text{-points}\}| &= q+1-b_2-\frac{q+1-t_2-b_2}{2} = \frac{1}{2}(q+1+t_2-b_2).
\end{align*}

\subsection*{Main formulas}

Putting these ingredients together, the total number of points in $\mathbb{P}^2(\mathbb{F}_q)$ that are external to $\mathcal C$ and internal to $\mathcal D$ is obtained by summing over all tangent lines $\mathcal L$ of $\mathcal C$ as 
\[
\frac{q - 1}{2} \times \text{(number of case (1) lines)} + \frac{q+1}{2}\times \text{(number of case (2) lines)}
  \;-\; (\mathbb{F}_q\text{-points on }\mathcal{C} \text{ internal to } \mathcal{D}) .
\]

We count the points external to $\mathcal{C}$ and internal to $\mathcal{D}$ twice, as there are two tangent lines to $\mathcal{C}$ passing such a point. Therefore, we divide the above count by $2$ to get the final formula. 

\begin{thm}[Main Counting Formula]\label{thm: main_conic_case}
Let $q$ be an odd prime power, and $\mathcal C, \mathcal D$ be two smooth plane conics defined over $\mathbb{F}_q$ with transversal intersection. Let \(E_1: y^2 = \det(D)\det(xC + D)\) and \(E_2: y^2 = \det(C)\det(xD + C)\) be the two elliptic curves associated to the pair of conics and their duals with Frobenius traces $t_1$ and $t_2$ respectively. Let $b_1,b_2$ be the number of $\mathbb{F}_q$-rational intersection points of the two conics, and of the two dual conics. Then,
\begin{align*}
    N_{\mathrm{ext\text{-}int}}(\mathcal C,\mathcal D) &= \frac{1}{4}\left(q^2-b_2q-1-t_1+b_1+t_2\right) ,\\ 
    N_{\mathrm{int\text{-}ext}}(\mathcal C,\mathcal D) &= \frac{1}{4}\left(q^2-b_2q-1+t_1+b_1-t_2\right) ,\\
    N_{\mathrm{ext\text{-}ext}}(\mathcal C,\mathcal D) &= \frac{1}{4}\left(q^2+b_2q-1+t_1+b_1+t_2\right) ,\\
    N_{\mathrm{int\text{-}int}}(\mathcal C,\mathcal D) &= \frac{1}{4}\left(q^2+b_2q-1-t_1+b_1-t_2-4q\right).
\end{align*}
In particular, we have 
$$\left|N_{\mathrm{int/ext\text{-}int/ext}}(\mathcal{C}, \mathcal{D}) - \frac{q^2}{4}\right|\leq q+\sqrt{q}+1.$$
\end{thm}

\begin{proof}
    With all of the above ingredients, we have 

    $$N_{\mathrm{ext\text{-}int}}(\mathcal C,\mathcal D)=\frac{1}{2}\left(\frac{q-1}{2}\frac{q+1-t_2-b_2}{2}+\frac{q+1}{2}\frac{q+1+t_2-b_2}{2}-\frac{q+1+t_1-b_1}{2}\right).$$
    We can simplify it as 
    $$N_{\mathrm{ext\text{-}int}}(\mathcal C,\mathcal D) = \frac{1}{4}\left(q^2-b_2q-1-t_1+b_1+t_2\right).$$

    The computation for $N_{\mathrm{ext\text{-}ext}}(\mathcal C,\mathcal D)$ is similar. The corresponding formula is given by 
    $$N_{\mathrm{ext\text{-}ext}}(\mathcal C,\mathcal D) = \frac{1}{2}\left(\frac{q-1}{2}\frac{q+1-t_2-b_2}{2} + \frac{q+1}{2}\frac{q+1+t_2-b_2}{2} + qb_2-\frac{q+1-t_1-b_1}{2}\right),$$
    which simplifies to
    $$N_{\mathrm{ext\text{-}ext}}(\mathcal C,\mathcal D) = \frac{1}{4}(q^2+b_2q-1+t_1+b_1+t_2).$$

    We note that $N_{\mathrm{int\text{-}ext}}(\mathcal{C}, \mathcal{D}) = N_{\mathrm{ext\text{-}int}}(\mathcal{D}, \mathcal{C})$. Therefore, the formula for $N_{\mathrm{int\text{-}ext}}(\mathcal{C}, \mathcal{D})$ follows from the formula for $N_{\mathrm{ext\text{-}int}}(\mathcal{C}, \mathcal{D})$ by swapping $t_1$ with $t_2$. 

    Lastly, we know that 
    $$N_{\mathrm{int\text{-}ext}}(\mathcal{C}, \mathcal{D}) + N_{\mathrm{int\text{-}int}}(\mathcal{C}, \mathcal{D}) + N_{\mathrm{ext\text{-}int}}(\mathcal{C}, \mathcal{D}) + N_{\mathrm{ext\text{-}ext}}(\mathcal{C}, \mathcal{D}) = q^2+q+1-(2q+2-b_1),$$
    where the right hand side is the number of $\mathbb{F}_q$-rational points of $\mathbb{P}^2\setminus (\mathcal C \cup \mathcal D )$. Then, the formula for $N_{\mathrm{int\text{-}int}}(\mathcal{C}, \mathcal{D})$ follows from the previous three computations. Since $0\leq b_i\leq 4$ and $|t_i|\leq 2\sqrt{q}$ from the Hasse bound, we get $q+\sqrt{q}+1$ as the upper bound for the error term. 
\end{proof}

\begin{rmk}[Comparison of the Error Terms]
The formula presented in Theorem \ref{thm: main_conic_case} improves the error term when compared to the multi-quadratic cover framework studied by Slavov \cite{Slavov}. If we approach this planar conic case using the multi-quadratic cover framework, the condition of being simultaneously external to $\mathcal{C}$ and internal to $\mathcal{D}$ corresponds to a specific Frobenius conjugacy class in a finite étale Galois cover $Y \to U \subset \mathbb{P}^2$ with Galois group $(\mathbb{Z}/2\mathbb{Z})^2$ \cite{Slavov, Mypaper}. The error term would be $O(q^{2-1/2}) = O(q^{3/2})$ coming from the Chebotarev density theorem or Lang-Weil estimate.

Our counting formula:
\[
N_{\mathrm{ext\text{-}int}}(\mathcal{C}, \mathcal{D}) = \frac{1}{4}\left(q^2 - b_2 q - 1 - t_1 + b_1 + t_2\right)
\]
demonstrates that the error term in the planar conic case is $O(q)$, and the errors $t_1, t_2$ coming from the cohomology are $O(q^{1/2})$ instead of $O(q^{3/2})$.

This reduction of the cohomological error from $O(q^{3/2})$ down to $O(q^{1/2})$ could be unique to the plane conic case, and may not generalize to higher-dimensional quadrics. It relies on the property that every $\mathbb{F}_q$-rational point on an $\mathbb{F}_q$-tangent line to the planar conic $\mathcal{C}$ is external to $\mathcal{C}$. By slicing the plane via the $\mathbb{F}_q$-tangent lines of $\mathcal{C}$, we restrict to points that are automatically external to $\mathcal{C}$. This transforms a multi-quadratic cover over $\mathbb{P}^2$ problem into a problem over a one-dimensional family of lines. This is why the cohomological error term reduces from $O(q^{3/2})$ to $O(q^{1/2})$ in this planar conic case.

For general higher-dimensional quadrics, a rational point lying on an $\mathbb{F}_q$-rational tangent space is not automatically external to the quadric. Under the multi-quadratic cover framework established in \cite{Slavov, Mypaper}, there is $O(q^{n-1})$ error term coming from the ramification locus, and a $O(q^{n-1/2})$ error term coming from the cohomology. The exact counting formula would be hard, as it relies on computation of higher étale cohomology groups.
\end{rmk}

\section*{Models for the incidence varieties \(E_1\) and \(E_2\)}
Now, we show that the incidence varieties 
$$E_1 = \{(P, \mathcal L)\in \mathcal C\times\mathcal D^* : P\in \mathcal L\}, \quad E_2 = \{(\mathcal L, P)\in\mathcal{C}^*\times\mathcal{D}: P\in\mathcal{L}\}$$
can be represented over $\mathbb F_q$ by the models 
$$ E_1: y^2 = \det(D)\det(xC+D), \quad 
    E_2: y^2 = \det(C)\det(xD+C). $$

The classical work by Griffiths and Harris \cite{GH} and a more explicit computation by Ionas \cite{Ionas} give an explicit model of this incidence variety over $\mathbb{C}$. They identify $E_1$ with the Cayley cubic $y^2 = \det(xC+D)$ over $\mathbb{C}$. While their result still holds for any algebraically closed field with characteristic other than $2$, the incidence variety $E_1$ is only isomorphic to a quadratic twist of $y^2 = \det(xC+D)$ over a non-algebraically closed field. The purpose of this section is to determine this twist over $\mathbb F_q$.

\begin{lem}\label{lem:incidence-cayley-twist}
Let $q$ be an odd prime power, and let $\mathcal C,\mathcal D\subset \mathbb P^2$ be smooth plane conics defined over $\mathbb{F}_q$ with matrix representations $C$ and $D$. We further assume that $\mathcal C$ and $\mathcal D$ intersect transversely over $\overline{\mathbb{F}_q}$. Let
$$E_1 =\{(P, \mathcal L)\in \mathcal C\times \mathcal D^*:P\in \mathcal L\}$$
be the incidence variety. Then, $E_1$ is isomorphic over $\mathbb{F}_q$ to the elliptic curve 
$$y^2 = \det(D)\det(xC+D).$$

\end{lem}

\begin{proof}
The projection $E_1\to \mathcal C$ via $(P, \mathcal L)\mapsto P$ is a 2-to-1 cover ramified at the four points of  $\mathcal{C}\cap \mathcal{D}$.  By Riemann-Hurwitz formula, we know that $E_1$ is a smooth curve with genus $1$. It remains to identify its Jacobian.

Since every smooth conic over $\mathbb F_q$ has a rational point, after a projective change of coordinates over $\mathbb{F}_q$, we may assume that $\mathcal{C}$ is defined by 
$$\mathcal C:\ XZ-Y^2=0.$$
Then $\mathcal C$ is parametrized as $\mathbb{P}^1$ via
$$(u:v)\mapsto (u^2:uv:v^2).$$

Write
$$
D=
\begin{pmatrix}
a&b/2&c/2\\
b/2&d&e/2\\
c/2&e/2&f
\end{pmatrix},
$$
so that
$$Q_D(X,Y,Z)=aX^2+bXY+cXZ+dY^2+eYZ+fZ^2.$$
Let
$$\delta=\det(D).$$
Explicitly,
$$4\delta=4adf+bce-ae^2-b^2f-c^2d.$$

Under the above parametrization of $\mathcal C$, we obtain the binary quartic
$$
f(u,v)=Q_D(u^2,uv,v^2)
=au^4+bu^3v+(c+d)u^2v^2+euv^3+fv^4.
$$
Its zeros are precisely the points of $\mathcal C\cap\mathcal D$.

For $P\notin\mathcal D$, the two tangent lines from $P$ to $\mathcal D$ are
defined over the quadratic extension determined by the discriminant
$-\det(D)Q_D(P)$, up to a nonzero square factor (see
\cite[Lemma~2.4]{Asgarli}). Hence the double cover $E_1\to\mathcal C\simeq
\mathbb P^1$ is represented by the binary quartic model
$$W^2=-\delta f(u,v).$$

We now compute its Jacobian using the $a_i$-invariants of Cremona--Fisher--Stoll \cite{cremona2010minimisation}. Their generalized binary quartic (\cite[Definition~2.2]{cremona2010minimisation}) has the form
$$y^2+P(x,z)y=Q(x,z),$$
where
$$P(x,z)=lx^2+mxz+nz^2$$
and
$$Q(x,z)=Ax^4+Bx^3z+Cx^2z^2+Dxz^3+Ez^4.$$
By \cite[Lemma 2.9, Theorem~2.10]{cremona2010minimisation}, the associated Weierstrass
equation for the Jacobian is
$$y^2+a_1xy+a_3y=x^3+a_2x^2+a_4x+a_6,$$
where
\begin{align*}
a_1&=m,\\
a_2&=C-ln,\\
a_3&=lD+nB,\\
a_4&=-4AE+BD-(l^2E+lnC+n^2A),\\
a_6&=-4ACE+AD^2+B^2E-(l^2CE+m^2AE+n^2AC+lnBD)+lmBE+mnAD.
\end{align*}
In our case $P=0$, so $l=m=n=0$, and
$$
Q(u,v)=-\delta f(u,v).
$$
Thus
$$
A=-\delta a,\quad B=-\delta b,\quad C=-\delta(c+d),\quad D=-\delta e,\quad
E=-\delta f.
$$
Substituting these values into the CFS formulas gives
$$a_1=0,\quad a_3=0,$$
$$a_2=-\delta(c+d),$$
$$a_4=\delta^2(-4af+be),$$
and
\begin{align*}
a_6
&=-4ACE+AD^2+B^2E\\
&=-4(-\delta a)(-\delta(c+d))(-\delta f)
  +(-\delta a)(-\delta e)^2+(-\delta b)^2(-\delta f)\\
&=\delta^3\left(4a(c+d)f-ae^2-b^2f\right).
\end{align*}
Therefore, by \cite[Theorem~2.10]{cremona2010minimisation}, over any field of
odd characteristic, including characteristic $3$, the Jacobian of the binary
quartic model $W^2=-\delta f(u,v)$ is the elliptic curve
\begin{equation}\label{eq:jacobian-weierstrass-model}
y^2=x^3-\delta(c+d)x^2+\delta^2(-4af+be)x
+\delta^3\left(4a(c+d)f-ae^2-b^2f\right).
\end{equation}

It remains to compare this Weierstrass equation with the determinantal cubic
$$Y^2Z=\delta\det(XC+ZD).$$
A direct determinant computation gives
\begin{align*}
\det(XC+ZD)=\frac{1}{4}X^3+\frac{2c-d}{4}X^2Z+\frac{-4af+be+c^2-2cd}{4}XZ^2+\delta Z^3.
\end{align*}
On the affine chart $Z=1$, the determinantal cubic is therefore
$$
Y^2=
\frac{\delta}{4}X^3+\frac{\delta(2c-d)}{4}X^2
+\frac{\delta(-4af+be+c^2-2cd)}{4}X+\delta^2.
$$
We claim that this curve is isomorphic to \eqref{eq:jacobian-weierstrass-model}
by the change of variables
$$x=\delta(X+c),\quad y=2\delta Y.$$
Indeed, multiplying the determinantal cubic equation by $4\delta^2$ gives
$$
(2\delta Y)^2
=
\delta^3\left(X^3+(2c-d)X^2+(-4af+be+c^2-2cd)X+4\delta\right).
$$
On the other hand, substituting $x=\delta(X+c)$ into the right hand side of
\eqref{eq:jacobian-weierstrass-model} gives
\begin{align*}
&\delta^3\Big((X+c)^3-(c+d)(X+c)^2+(-4af+be)(X+c)\\
&\hspace{4.5cm}+4a(c+d)f-ae^2-b^2f\Big).
\end{align*}
Using
$$4\delta=4adf+bce-ae^2-b^2f-c^2d,$$
the expression inside the parentheses simplifies to
$$X^3+(2c-d)X^2+(-4af+be+c^2-2cd)X+4\delta.$$
Thus the two affine equations are identified by
$$x=\delta(X+c),\quad y=2\delta Y.$$
This change of variables is defined over $\mathbb F_q$ because $q$ is odd and
$\delta=\det(D)\neq 0$. Hence the Jacobian of $E_1$ is isomorphic over
$\mathbb F_q$ to the smooth projective model of
$$Y^2Z=\delta\det(XC+ZD).$$

Finally, every genus-one curve over a finite field has a rational point. Hence
$E_1$ is isomorphic over $\mathbb F_q$ to its Jacobian. Therefore $E_1$ is
isomorphic over $\mathbb F_q$ to the smooth projective model of
$$Y^2Z=\det(D)\det(XC+ZD).$$
\end{proof}

We now obtain the corresponding model for $E_2$ by applying the lemma to the
dual conics. Since $\mathcal C^*$ and $\mathcal D^*$ are represented, up to
nonzero scalar, by $C^{-1}$ and $D^{-1}$, the preceding lemma gives the model
$$
y^2=\det(D^{-1})\det(xC^{-1}+D^{-1}).
$$
Using
$$
\det(D^{-1})\det(xC^{-1}+D^{-1})
=(\det C)^{-1}(\det D)^{-2}\det(xD+C),
$$
we see that this curve is isomorphic over $\mathbb F_q$ to
$$y^2=\det(C)\det(xD+C),$$
because $(\det C)^{-1}(\det D)^{-2}$ and $\det(C)$ differ by the square
$(\det C\det D)^{-2}$ in $\mathbb F_q^\times$. Thus $E_2$ can be modeled over
$\mathbb F_q$ by
$$E_2:y^2=\det(C)\det(xD+C).$$

\section{Sage Worksheet}
We also provide a Sage worksheet for the verification of the formula in Theorem \ref{thm: main_conic_case} at \url{https://github.com/TianhaoW/An-Exact-Counting-Formula-for-the-Mutual-Position-of-Two-Plane-Conics}. 

\section*{Acknowledgments}
I would like to thank my advisor, Nathan Kaplan, for suggesting the question considered in this note and for many helpful discussions. I would also like to thank Shamil Asgarli and Chi Hoi Yip for helpful correspondence, and for kindly providing a copy of the paper of Abatangelo, Fisher, Korchmáros, and Larato.

The author used GPT-5.5 for assistance with grammar, spelling, Sage worksheet, and the organization of symbolic computations appearing in the identification of the incidence varieties with the elliptic curve models. The author independently checked the computations and takes full responsibility for the final manuscript.

\newpage
\printbibliography

@article{Slavov,
  title = {Square values of several polynomials over a finite field},
  author = {Slavov, Kaloyan},
  journal = {Finite Fields and Their Applications},
  volume = {109},
  pages = {102696},
  year = {2026},
  issn = {1071-5797},
}

@article{GH,
  title = {On {C}ayley's explicit solution to {P}oncelet's porism},
  author = {Griffiths, Phillip and Harris, Joseph},
  journal = {L'Enseignement Math{\'e}matique},
  series = {2},
  volume = {24},
  number = {1-2},
  pages = {31--40},
  year = {1978},
  publisher = {L'Enseignement Math{\'e}matique}
}

@article{Asgarli,
  title={Mutual position of two smooth quadrics over finite fields},
  author={Asgarli, Shamil and Yip, Chi Hoi},
  journal={Designs, Codes and Cryptography},
  volume={93},
  number={10},
  pages={4461--4472},
  year={2025},
  publisher={Springer}
}

@misc{Ionas,
      title={Elliptic constructions of hyperkaehler metrics {III}: Gravitons and {P}oncelet polygons}, 
      author={Radu A. Ionas},
      year={2007},
      eprint={0712.3601},
      archivePrefix={arXiv},
      primaryClass={math.DG},
}

@article{Korchmros,
  title={Problems and results in {$PG(2, q)$}},
  author={G{\'a}bor Korchm{\'a}ros},
  journal={Electron. Notes Discret. Math.},
  year={2013},
  volume={40},
  pages={181-187},
}

@misc{Mypaper,
      title={Splitting of Polynomial Families via {G}alois Theory}, 
      author={Tianhao Wang},
      year={2026},
      eprint={2606.12810},
      eprinttype={arXiv},
      eprintclass={math.NT},
}

@article{abatangelo2011mutual,
  author    = {Abatangelo, V. and Fisher, J. C. and Korchm{\'a}ros, G. and Larato, B.},
  title     = {On the mutual position of two irreducible conics in $\mathrm{PG}(2,q)$, $q$ odd},
  journal   = {Advances in Geometry},
  volume    = {11},
  number    = {4},
  pages     = {603--614},
  year      = {2011},
  publisher = {De Gruyter},
  doi       = {10.1515/advg.2011.026}
}

@book{harrisag,
  title={Algebraic Geometry: A First Course},
  author={Harris, Joe},
  volume={133},
  series={Graduate Texts in Mathematics},
  year={1995},
  publisher={Springer-Verlag},
  address={New York, NY}
}

@article{cremona2010minimisation,
  title={Minimisation and reduction of 2-, 3- and 4-coverings of elliptic curves},
  author={Cremona, John E and Fisher, Tom A and Stoll, Michael},
  journal={Algebra \& Number Theory},
  volume={4},
  number={6},
  pages={763--820},
  year={2010},
  publisher={Mathematical Sciences Publishers}
}

\end{document}